\documentclass[letterpaper,reqno]{amsart}
\usepackage{graphicx}

\pdfoutput=1

\usepackage{amssymb,mathrsfs, amsmath,amsfonts,verbatim}
\usepackage{mathtools}
\usepackage{bm}
\usepackage[usenames]{color}

\usepackage{scalerel}

\usepackage{xfrac}

\usepackage{enumerate}

\usepackage{tikz}

\usepackage{microtype}
\usepackage{stackengine}
\usepackage{cmap}
\usepackage{bbm}

\usepackage[colorlinks=true,allcolors=blue]{hyperref}

\newtheorem{theorem}{Theorem}[section]
\newtheorem{corollary}[theorem]{Corollary}
\newtheorem{lemma}[theorem]{Lemma}
\newtheorem{proposition}[theorem]{Proposition}

\theoremstyle{definition}
\newtheorem{definition}[theorem]{Definition}
\newtheorem{remark}[theorem]{Remark}
\theoremstyle{definition}

\newtheorem{fact*}{Fact}

\newcommand{\finv}[1]{{#1}^{\bm{\dagger}/2}}
\newcommand{\alg}{\mathrm{alg}}
\DeclareMathOperator\vecc{{\bf vec}}

\newcommand{\B}{\mathcal{B}}

\newcommand{\M}{\mathcal{M}}
\newcommand{\cM}{\mathcal{M}}
\newcommand{\Z}{\mathbb{Z}}

\newcommand{\cE}{\mathcal{E}}

\newcommand{\cQ}{\mathcal{Q}}

\newcommand{\cV}{\mathcal{V}}

\newcommand\BH{B(\mathcal{H})}

\newcommand{\C}{\mathbb{C}}

\newcommand{\bbx}{\mathbbm{x}}

\newcommand\Choi{\mathfrak{C}}

\newcommand{\NP}[1]{\mathrm{NP}(#1)}
\newcommand{\ANP}[1]{\mathrm{ANP}(#1)}
\newcommand{\rowt}{\mathrm{row}}
\newcommand{\colt}{\mathrm{col}}

\newcommand\ep{\varepsilon}

\newcommand{\set}[1]{{\left\{#1\right\}}}

\newcommand{\norm}[1]{{\left\Vert#1\right\Vert}}
\newcommand{\snorm}[1]{{\|#1\|}}

\newcommand{\trp}{\bm{\tau}}
\newcommand{\ran}[1]{\operatorname{ran}#1}

\newcommand{\til}{\raise.17ex\hbox{$\scriptstyle\mathtt{\sim}$}}

\newcommand{\bbm}{\left[ \begin{smallmatrix}}
\newcommand{\ebm}{\end{smallmatrix} \right]}
\newcommand{\bBm}{\left[ \begin{matrix}}
\newcommand{\eBm}{\end{matrix} \right]}

\newcommand{\bpm}{\begin{pmatrix}}
\newcommand{\epm}{\end{pmatrix}}
\newcommand{\bal}{\begin{align*}}
\newcommand{\eal}{\end{align*}}

\newcommand\smallmath[2]{#1{\raisebox{\dimexpr \fontdimen 22 \textfont 2
      - \fontdimen 22 \scriptfont 2 \relax}{$\scriptstyle #2$}}}
\newcommand{\sot}{\!\smallmath\mathbin\otimes\!}

\numberwithin{equation}{section}

\newlength{\Mheight}
\newlength{\cwidth}

\newcommand{\fralg}[1]{\langle #1 \rangle}
\newcommand{\fax}{\fralg{\bbx}}

\showboxdepth=\maxdimen
\showboxbreadth=\maxdimen

\makeindex

\allowdisplaybreaks

\title[Effective nc Pick interpolation]{Effective noncommutative Nevanlinna-Pick interpolation in the row ball, and applications}

\author{Meric Augat}
\address{Washington University in St. Louis}
\email{maugat@wustl.edu}

\author{Michael T. Jury${}^*$}
\address{University of Florida}
\email{mjury@ufl.edu}

\author{James Eldred Pascoe}
\address{University of Florida}
\email{pascoej@ufl.edu}

\thanks{${}^*$Research Supported by NSF grant DMS-1900364.}

\thanks{}

\begin{document}

\begin{abstract}
We provide an effective single-matrix criterion, in terms of what we call the {\em elementary Pick matrix}, for the solvability of the noncommutative Nevanlinna-Pick interpolation problem in the row ball, and provide some applications. In particular we show that the so-called ``column-row property'' fails for the free semigroup algebras, in stark contrast to the analogous commutative case. Additional applications of the elementary Pick matrix include a local dilation theorem for matrix row contractions and interpolating sequences in the noncommutative setting. Finally we present some numerical results related to the failure of the column-row property. 
\end{abstract}

\maketitle

\section{Introduction}
\label{sec:intro}
\subsection{} The purpose of this paper is to give an effective solution of the so-called ``noncommutative Nevanlinna-Pick interpolation 
problem'' in the row ball, which is an analog, in the modern setting of noncommutative function theory, of the classical Nevanlinna-Pick 
interpolation problem.  The main result is the construction of a single matrix, in closed form, such that the problem has a solution if and 
only if this matrix is positive semidefinite. In this introductory section we pose the problem and describe some of the applications of our 
solution.


\subsection{Noncommutative Pick interpolation in the row ball} 

We work in the general setting of noncommutative function theory, as laid out e.g. in \cite{KVV-tome}. Fix an integer $d\geq 1$. For each $n=1, 2, 3, \dots, $ let $\cM_n^d$ denote the set of $d$-tuples of $n\times n$ matrices with 
complex entries:
\[
	\cM_n^d =\{X=(X_1, \dots, X_d): X_i \in \cM_n\}
\]
and let $\cM^d$ be the disjoint union of the $\cM_n^d$ over all $n$ (When $d=1$ we drop the superscripts and just write $\cM_n, 
\cM$). 
Let $\cM_{s\times t}$ denote the set of $s\times t$ matrices with complex entries.
By the {\em row ball $\mathcal B^d$} we mean the graded subset of $\cM_d$, defined at each ``level'' $n$ by
\[
\B_n^d =\{X=(X_1, \dots , X_d) \in \cM_n^d: \|X_1 X_1^*+\cdots +X_dX_d^*\|<1\} \subset \cM_n^d.
\]
The row ball $\mathcal B^d$ is a prototypical example of an {\em nc domain}; this means that
(1) at each level, the set $\B_n^d \subset \cM_n^d$ is open, (2)  $\mathcal B^d$ respects direct sums, i.e. if $X\in \mathcal B^d_n$ and $Y\in \mathcal B^d_m$ then $X\oplus Y \in \mathcal B^d_{m+n}$ 
(here the direct sum means coordinatewise direct sum: $X\oplus Y = (X_1\oplus Y_1, \dots, X_d\oplus Y_d)$; and (3) $\mathcal B^d$ respects 
unitary equivalence, i.e. if $U\in \M$ is a unitary matrix and $X=(X_1, \dots, X_d)\in \mathcal B^d_n$ then $U^*XU = (U^*X_1U, 
\dots , U^*X_dU)\in\mathcal B^d_n$.  

The nc-domain $\mathcal B^d$ then supports {\em nc-functions}, which are graded functions $f:\mathcal 
B^d \to \cM$ 
(that 
is, a family of functions $f_n:\mathcal B^d_n\to \cM_n$, $n=1, 2, 3, \dots$ which (1) respect direct sums: for $X\in \mathcal B^d_n$, 
$Y\in\mathcal B^d_m$, we have $f_{m+n}(X\oplus Y)=f_n(X)\oplus f_m(Y)$; and (2) respect similarities, in the sense that if $X\in \mathcal B^d$ and $S$ is a similarity such that $S^{-1}XS$ is also in $\mathcal B^d$, then $f(S^{-1}XS) =S^{-1}f(X)S$. Let 
\[
H^\infty(\mathcal B^d) =\{f:\mathcal B^d\to \cM: f \text{ is an nc function and } \sup_{X\in\mathcal B^d} \|f(X)\|<\infty\}.
\]
We refer to the supremum in this definition as the $H^\infty$ norm of the nc function $f$, denoted $\snorm{f}_\infty$.

The {\em noncommutative Nevanlinna-Pick interpolation problem} in the row ball is the following (see \cite{BMV18} and the references therein) : given a finite set of points (``nodes'') 
$X^1, \dots , X^m$ in $\mathcal B^d$, with $X^j\in \mathcal B^d_{n_j}$, and matrices $Y^1, \dots Y^m$, with $Y^j\in \cM_{n_j}$, find an 
interpolating function $f\in H^\infty (\mathcal B^d)$ (if it exists)
\begin{equation}\label{eqn:first-NP-problem-statement}
f(X^j)=Y^j \quad j=1, \dots m
\end{equation}
of minimal $H^\infty$ norm. The fact that the domain $\mathcal B^d$ and the nc functions $f$ respect direct sums means that every such problem can be immediately reduced to a ``one-point problem'': putting $X=\oplus X^j$ and $Y=\oplus Y^j$, the problem (\ref{eqn:first-NP-problem-statement}) has a solution if and only if the one-point problem
\begin{equation}\label{eqn:NP-onepoint-statement}
f(X)=Y,
\end{equation}
has a solution, and the minimal norms are the same.   Instead of asking for the minimal norm, one could pose the essentially equivalent problem of asking whether or not there exists a solution of norm $\|f\|_\infty\leq 1$. It is also possible to consider a generalized problem in which the single $n\times n$ matrix $Y$ is replaced by an $s\times t$ block matrix $(Y_{ij}), i=1, \dots, s; j=1, \dots, t$, where each $Y_{ij}$ is an $n\times n$ matrix. We then seek an $s\times t$ matrix of nc functions $F=(f_{ij})$ so that
\[
f_{ij}(X)=Y_{ij} \quad i=1, \dots, s; j=1, \dots, t
\]
and the $H^\infty$ norm of the $s\times t$ matrix nc function $F$ is the evident supremum norm. 

When $d=1$ and all the $X^j,Y^j$ are $1\times 1$ matrices this reduces to the classical Nevanlinna-Pick interpolation problem in the unit 
disk. In that case, interpolating functions always exists (e.g. one can take a Lagrange interpolating polynomial), so the problem is just one 
of finding the minimal $H^\infty$ norm. However in the noncommutative setting solutions need not always exist; a necessary and sufficient 
condition for a solution of the one-point problem \eqref{eqn:NP-onepoint-statement} is that the matrix $Y$ belong to the subalgebra of 
$\cM_n$ generated by the coordinates $X_1, \dots, X_d$ of the point $X$. 

Consider for a moment the classical Nevanlinna-Pick interpolation problem: given points $x^1, \dots, x^m$ in the open unit disk, and complex numbers $y^1, \dots, y^m$, does there exist an analytic function $f$, bounded by $1$ in the disk, with
\begin{equation}\label{eqn:NP-classical}
f(x^j)=y^j, \quad j=1, \dots, m?
\end{equation}
The problem has a solution if and only if the {\em Pick matrix}
\[
P= \left( \frac{1-y^i\overline{y^j}}{1-x^i\overline{x^j}}\right)_{i,j=1}^m
\]
is positive semidefinite. It turns out that it is also possible to give a necessary and sufficient condition for the existence of a norm-one 
solution of the noncommutative problem (\ref{eqn:NP-onepoint-statement}) in terms of a single matrix involving the data $X,Y$, this was given 
by Ball, Marx, and Vinnikov in \cite{BMV18}; however the single matrix in question is expressed as an infinite sum and does not have a 
readily apparent closed form. The main result of the present paper is to present a closed-from expression for this ``noncommutative Pick 
matrix,'' which is amenable at least in some cases to machine computation, thus providing an effective solution to the problem which is numerically stable for suitably conditioned data. We construct this closed form expression in Section~\ref{sec:nc pick and matrix PX}, the key idea is a 
matrix involution introduced previously in \cite{pascoe-2019-alg} in connection with the problem of determining the algebra generated by a 
family of matrices $X_1, \dots, X_d$ (which is connected to the interpolation problem, as remarked above).  


\subsection{Failure of the column-row property in $\mathcal{L}^d$}

Let $\mathcal H$ be a Hilbert space, $\BH$ the algebra of bounded operators on $\mathcal H$, and fix a subset $\mathcal A \subseteq \BH.$ For each fixed $n\geq 1$, we define $C_n$ to be the least number $C_n$ such that the inequality
\[
	\Big\|\sum_{i=1}^n A_iA_i^*\Big\|^{1/2}\leq C_n	\Big\|\sum_{i=1}^n A_i^*A_i\Big\|^{1/2}
\]
holds for all $n$-tuples $A_1, \dots , A_n$ of elements from $\mathcal A$.  The {\bf column-row constant} of $\mathcal A$ is the least number $C$ such that 
\[
	\Big\|\sum_{i=1}^\infty  A_iA_i^*\Big\|^{1/2}\leq C \Big\|\sum_{i=1}^\infty A_i^*A_i\Big\|^{1/2} 
\]
for all sequences $(A_i)^\infty_{i=1}$ from $\mathcal A$ for which the sums are SOT-convergent.   Evidently the $C_n$ form an increasing 
sequence with $\lim C_n =C$; it is possible that $C=\infty$.  If $C$ is finite, we say that $\mathcal A$ has the {\bf column-row property}. 
(One could analogously define a row-column property but this will not concern us here.)
For example, $\mathcal A=M_n(\mathbb{C})$ has column-row constant at least equal to $\sqrt{n}.$  It is also easy to verify that for any set of operators $\mathcal A$, we have $C_n\leq \sqrt{n}$ for every $n$. 

Of particular interest is the case when $\mathcal A$ is the algebra of bounded multiplication operators on a reproducing kernel Hilbert 
space. In this setting a number of important spaces are known to have this property. Trivially, the algebra $\mathcal A=H^\infty(\mathbb D)$ 
(the algebra of bounded analytic functions in the unit disk $\mathbb D$, equipped with the supremum norm) has the column-row property. Beyond 
this, the multiplier algebra of the Dirichlet space $\mathcal D$ over the unit disk has the column-row property with constant $C\leq 
\sqrt{18}$ \cite{trent-2004}, and the multiplier algebras of the {\em Drury-Arveson spaces} $H^2_d$ over the unit ball $\mathbb B^d\subset 
\mathbb C^d$ (denoted $Mult(H^2_d)$) have the column-row property with constants $C=C(d)$; \cite{AHMR-2018}, in the proof given in 
\cite{AHMR-2018} the obtained estimates on the constants $C(d)$ grow to infinity with the dimension $d$. The Dirichlet space and the $H^2_d$ 
spaces are particular examples of spaces with a {\em complete Nevanlinna-Pick (CNP) kernel}, the column-row property (when it holds) turns 
out to have important consequences in such spaces, e.g. in applications to interpolating sequences \cite{AHMR-2019} and in factorization of 
weak products \cite{jury-martin-2019}, \cite{AHMR-2018}.  


The connection with the present paper is as follows: it turns out that the multiplier algebras of $H^2_d$ can be viewed as the ``commutative collapse'' of the so-called {\em free semigroup 
algebras} $\mathcal L_d$, $d\geq 2$. (We refer to the survey \cite{Dav01} for the basic facts about the free semigroup algebras.) One may then ask if an analog of the column-row property holds for these algebras. In detail, if we let $\mathbb F_d^+$ denote the free semigroup of all 
noncommuting words in $d$ letters  $\{1, 2, \dots, d\}$, (including the ``empty word'' $\varnothing$), then we can form a Hilbert space 
$\mathcal F^2_d$ with orthonormal basis $\{\xi_w\}_{w\in\mathbb F^+_d}$. For each letter $i$ we define an operator
\[
L_i\xi_w =\xi_{iw}, \quad w\in\mathbb F^+_d.
\]
The operators $L_i$ are isometries with orthogonal ranges, i.e. we have $L_i^*L_j=\delta_{ij} I$ for $i, j=1, \dots , d$. The free semigroup algebra is the WOT-closed algebra generated by the $L_i$, $i=1, \dots, d$. 

By a result of Popescu (\cite[Theorem 3.1]{popescu-2006}) the free semigroup algebra $\mathcal L_d$ may be completely isometrically 
identified with the algebra $H^\infty(\B^d)$ of bounded nc functions in the row ball. Moreover the map $f\to f(z)$ obtained by restricting an 
nc function to level 1 (the scalar unit ball $\mathbb B^d\subset \mathbb C^d$) is a completely contractive homomorphism from $H^\infty(\B^d)$ 
onto the multiplier algebra $Mult(H^2_d)$, (see \cite[Theorem 4.4.1, Subsection 4.9]{shalit-2013} or \cite[Section 
2]{davidson-pitts-NP-1998}; this latter reference makes clear the connection with Nevanlinna-Pick interpolation) .

In particular, we observe that for each $d$, and $n$, the column-row constants $C_n$ for $H^\infty(\B^d)$ dominate the corresponding 
constants for $Mult(H^2_d)$. The question naturally arises of whether or not the free semigroup algebras $H^\infty(\B^d)$ have the column-row 
property. It turns out they do not; in fact we will prove the constant is infinity for $H^\infty(\B^d),$ and the constant $C_n = \sqrt{n}.$

\begin{theorem}
	\label{thm:row-col fails for Fock}
For the algebra of bounded nc functions in the row ball, $H^\infty(\B^d)$, $d\geq 2$, we have $C_n=\sqrt{n}$ for all $n=1, 2, \dots$. 
\end{theorem} 
Thus, in contrast to $Mult(H^2_d)$, the column-row property fails in  $H^\infty(\B^d)$ in the strongest possible way, establishing a stark contrast between the commutative multiplier algebras $Mult(H^2_d)$ and their noncommutative ``parents.'' Theorem \ref{thm:row-col fails for Fock} is proved in Section
\ref{sec:failure of column-row}.

\subsection{Readers' guide} 
Section~\ref{sec:prelims} gives a definition of the $\psi$-involution first introduced in \cite{pascoe-2019-alg}. We use the 
$\psi$-involution liberally throughout Section~\ref{sec:nc pick and matrix PX} to first construct for a (contractive) matrix tuple $X = 
(X_1,\dots, X_d)\subset \cM_n^d$ its elementary Pick matrix: a matrix $P_X$ whose range encodes the unital subalgebra of $\cM_n$ generated by 
$X_1,\dots, X_d$. 
This in turn is used to establish two of the main results of the paper: Theorem~\ref{thm:effective NP} and Theorem~\ref{thm:ANPMainTheorem}.

Section~\ref{sec:boomerang matrix} consists of several technical results leading up to the construction of an isometry in 
Section~\ref{sec:mini-dilations} and its immediate use in Theorem~\ref{thm:mini-dilation}, a so-called ``mini-dilation."

In Section~\ref{sec:failure of column-row} we apply Theorem~\ref{thm:ANPMainTheorem} to prove
Theorem~\ref{thm:row-col fails for Fock}: the column-row property fails for the Fock space on two or more generators.

Section~\ref{sec:examples} gives a more concrete approach to the results in Section~\ref{sec:failure of column-row}. Section~\ref{sec:remarks 
on NPX norm} introduces a condition number for a matrix tuple $X = (X_1,\dots, X_d)$ and explores its properties and interpolating sequences.
Finally, Section~\ref{sec:numerics} discusses computational consequences of effective NP-interpolation.

%
%
%

\section{Preliminaries. The $\psi$ involution and its properties}
\label{sec:prelims}

If $A\in \cM_{n\times m}$ and $B\in \cM_{r\times s}$ then their {\bf Kronecker product} $A\otimes B\in M_{nr\times ms}$ is the block 
matrix given by
\begin{equation}\label{eqn:kronecker_product_defn}
	A\otimes B = 
		\bpm 
			a_{11}B & \dots & a_{1m}B \\ 
			\vdots	& \ddots & \vdots \\
			a_{n1}B & \dots & a_{nm}B
		\epm.
\end{equation}
Or, in other words, $(A\otimes B)_{n(i-1)+k, n(j-1)+ \ell} = A_{i,j}B_{k,\ell}$.

Let $\trp:\M\to\M$ be the transpose operator and let $\vecc:\cM_n\to \cM_{n^2\times 1}$ be the linear map taking the 
	columns of a matrix and stacking them to get a column vector:
\[
	\vecc \bpm
		a_{11} & \dots & a_{1n} \\
		\vdots & \ddots & \vdots \\
		a_{n1} & \dots & a_{nn}
	\epm
	=
	\bpm
		a_{11}\\
		\vdots\\
		a_{n1}\\
		a_{12}\\
		\vdots\\
		a_{nn}
	\epm.
\]
We have the classical identity 
\begin{equation}\label{eqn:classical_vec_identity}
	\vecc(AXB) = (B^T\otimes A)\vecc(X).
\end{equation}
Typically we treat $\vecc$ as a graded function on $\M$.
That is, $\vecc = (\vecc[n])_{n=1}^\infty$, where each $\vecc[n]:\cM_n \to \cM_{n^2\times 1}$, and if $A\in \M_n$, then $\vecc(A) = 
\vecc[n](A)$.
This greatly simplifies notation.

\begin{definition}
	Define $\psi:\cM_{n^2}\to \cM_{n^2}$ to be $\psi = (\trp\circ \vecc) \otimes \vecc$.
	If $A\in \cM_{n^2}$ then we write the evaluation of $\psi$ on $A$ as
	\[
		A^\psi = \big[(\trp\circ \vecc) \otimes \vecc\big](A).
	\]
	Or, more explicitly, if $C,D\in \cM_n$ then
	\begin{equation}\label{eqn:psi_working_defn}
		[C\otimes D]^\psi = \vecc(C)^T\otimes \vecc(D) = \vecc(D)\vecc(C)^T.
	\end{equation}
Indeed, observe
	\[
		[A\sot B]^\psi = [(\trp\circ \vecc)\sot \vecc](A\sot B) = \vecc(A)^T\sot \vecc(B) = \vecc(B)\vecc(A)^T.
	\]
\end{definition}

Our motivation for writing $\psi$ as a superscript is that $\psi$ is an involution on $\cM_{n^2}$:
	
	\begin{lemma}
	\label{lem:Psi involution}
	For any $E_{ij}$ and $E_{k\ell}$ we have
	\begin{equation}\label{eqn:psi_on_matrix_units}
		[E_{ij}\otimes E_{k\ell}]^\psi = E_{\ell j}\otimes E_{k i}.
	\end{equation}
	Consequently, $\psi$ is an involution.
\end{lemma}
	\begin{proof}
	We have the following equalities:
	\begin{align*}
		[E_{ij}\otimes E_{k\ell}]^\psi 
			&= \vecc(E_{k\ell})\vecc(E_{ij})^T\\
			&= E_{n(\ell-1)+k, n(j-1)+i}\\
			&= E_{\ell j}\otimes E_{k i}.
	\end{align*}
Evidently applying $\psi$ again gives us back $E_{ij}\otimes E_{k\ell}$. Therefore $\psi$ is an involution.
\end{proof}


\begin{proposition}[$\psi$ modularity]
	\label{prop:modularity of psi}
	If $U\in \cM_{n^2}$ and $A,B,C,D\in \cM_n$ then
	\[
		\Big[(A\otimes B)U(C\otimes D)\Big]^\psi = (D^T\otimes B)U^\psi (C\otimes A^T).
	\]
\end{proposition}

\begin{proof}
	First we recall that if $u,v$ are column vectors, then $u^T\otimes v = vu^T$.
	We first prove the result for $U=E_{ij}\otimes E_{k\ell}$. Using (\ref{eqn:classical_vec_identity}) and (\ref{eqn:psi_working_defn}), we have 
	\begin{align*}
	\Big[(A\otimes B)(E_{ij}\otimes E_{k\ell})(C\otimes D)\big]^\psi
			&=\Big[(AE_{ij}C)\otimes (BE_{k\ell}D)\Big]^\psi\\
			&= \vecc(BE_{k\ell}D) (\vecc(AE_{ij}C) )^T\\
			&= \left[(D^T\otimes B)\vecc(E_{k\ell})\right]\left[(C^T\otimes A)\vecc(E_{ij})\right]^T\\
			&= (D^T\otimes B)\vecc(E_{k\ell})\vecc(E_{ij})^T(C\otimes A^T)\\
			&= [D^T\otimes B][E_{ij}\otimes E_{k\ell}]^\psi[C\otimes A^T].
	\end{align*}
	Since $\psi$ is linear and the $E_{ij}\otimes E_{k\ell}$ form a basis for $\cM_{n^2}$, we are done.
\end{proof}

\begin{remark}
	We could just as easily use Equation~\eqref{eqn:psi_on_matrix_units} as the definition of the $\psi$-involution. The $\psi$-involution 
	was introduced by the third named author in \cite{pascoe-2019-alg}, where its key properties (including the modularity property) were 
	described; we have included proofs here for the sake of convenience. What we will call the {\em elementary Pick matrix}, defined in the 
	next section, also appears in \cite{pascoe-2019-alg}. 
\end{remark}

%
%
%

\section{Noncommutative Pick Interpolation and the matrix $P_X$}
\label{sec:nc pick and matrix PX}

Recall $X = (X_1,\dots, X_d)\in \cM_n^d$ is a {\bf row contraction} if $[\begin{matrix} X_1 & \dots & X_d \end{matrix}]$ has norm strictly 
less than $1$:
\[
	\Big\| \sum_{i=1}^d X_iX_i^* \Big\| <1.
\]
\begin{definition} 
	For $X = (X_1,\dots, X_d)\in \cM_n^d$, we put 
\[	
	P_X := \Big[(I_n \otimes I_n - \sum_{i=1}^d \overline{X_i}\otimes X_i)^{-1}\Big]^\psi
\]
(when it is defined).  If $X$ is a row contraction, then it follows from \cite[ Proposition 3.1]{pascoe-2019-QPF} that the spectral radius of $ \sum_{i=1}^d \overline{X_i}\otimes X_i $ is strictly less than $1$, so $P_X$ exists.	In this case we call $P_X$ the {\bf elementary Pick matrix}. 
\end{definition}

\begin{definition}
	For $\bbx = \set{x_1,\dots, x_d}$, a set of freely noncommuting indeterminates, let $\fax = \fralg{x_1,\dots, x_d}$ 
		denote the unital free semigroup generated by $x_1,\dots, x_d$ with empty product $\varnothing$ acting as the identity. If $w=i_1i_2\dots i_n$ is a word in the letters $\{1, 2, \dots, d\}$ we write
		\[
		x^w := x_{i_1}x_{i_2}\cdots x_{i_n}.
		\]
In particular, for a system of matrices $X=(X_1, \dots, X_d)$ and a word $w$ we write
\[
	X^w := X_{i_1}X_{i_2}\cdots X_{i_n}.
\]
\end{definition}
Thus, when $X$ is a row contraction we can express $P_X$ as a norm-convergent power series
\[
P_X= \sum_{n=0}^\infty \left(\sum_{i=1}^d \overline{X_i}\otimes X_i \right)^n =\sum_{w\in \fax} {\overline{X}}^w\otimes X^w.
\]

Suppose now $X$ is a row contraction.
We recall the one-point nc Pick interpolation problem from the introduction: given $Y = (Y_{i,j})\in \cM_{s\times t}\otimes \cM_n$, does 
there exist an nc function $f\in \cM_{s\times t}\otimes H^\infty(\B^d)$ such that $\norm{f}_\infty\leq 1$ and $f(X) = Y$?

From \cite[Theorem 6.5]{BMV18}, this problem has a solution if and only if the map $\Phi:\cM_n\to \cM_{ns}$ 
\[
	\Phi(H) = \sum_{w\in \fax} (X^w H X^{w*})\otimes I_s - Y\left(\sum_{w\in \fax}(X^w H X^{w*})\otimes I_t\right)Y^*
\]
is completely positive.  Our goal is to recast this condition in terms of the elementary Pick matrix $P_X$ introduced above. To do this we first apply Choi's criterion to reduce the problem of checking the complete positivity of $\Phi$ to checking the positivity of a single matrix. We then use the $\psi$ involution to express this single matrix in closed form.

\begin{definition}
	For each $n$, the {\bf Choi Matrix} is the matrix 
\begin{equation}\label{eqn:choi_matrix_defn}
\mathfrak{C}_n = \sum_{i,j=1}^n E_{ij}\otimes E_{ij} \in \cM_{n^2}
\end{equation}
	By Choi's Theorem (see e.g. \cite[Theorem 3.14]{paulsen-book}), a map $\Phi:\cM_n\to \cM_m$ is completely positive if and only if the 
	single $nm\times nm$ matrix
	\[
		(I_n\otimes \Phi)(\Choi_n) = \sum_{i,j} E_{ij} \otimes \Phi(E_{ij})
	\]
is positive semidefinite. 	

The Choi Matrix also has the following important relation with $\psi$:
	\begin{equation}\label{eqn:choi_psi_identity}
		[I_{n^2}]^\psi = \mathfrak{C}_n,
	\end{equation}
as is trivially verified using (\ref{eqn:psi_on_matrix_units}) and (\ref{eqn:choi_matrix_defn}).
\end{definition}

\begin{lemma}
	\label{lem:Choi to P}
	If $X = (X_1,\dots, X_d)\in \cM_n^d$ is a row contraction then
	\[
		\sum_{w\in \fax} (I \otimes X)^w \mathfrak{C}_n (I\otimes X)^{w*} = P_X.
	\]
\end{lemma}

\begin{proof} Since $X$ is a row contraction, the series is norm convergent. Using the fact that $\psi$ is an involution, the modularity property (Proposition~\ref{prop:modularity of psi}), and the action of $\psi$ on the Choi matrix (\ref{eqn:choi_psi_identity}), we have
	\begin{align*}
		\sum_{w\in \fax} (I \otimes X)^w \mathfrak{C}_n (I\otimes X)^{w*}
			&= \left[{\textstyle\sum}_{w\in \fax}\left[(I\otimes X^w) \mathfrak{C}_n (I\otimes X^{w*})\right]^\psi\right]^\psi\\
			&= \left[{\textstyle\sum}_{w\in \fax}(\overline{X}^w\otimes X^w) I_{n^2} (I\otimes I)\right]^\psi\\
			&= \left[{\textstyle\sum}_{w\in \fax}(\overline{X}\otimes X)^w\right]^\psi\\
			&= \left[\left(I\otimes I - {\textstyle\sum}_{i}\overline{X_i}\otimes X_i\right)^{-1}\right]^\psi\\
			&= P_X
	\end{align*}
\end{proof}

\begin{theorem}
	\label{thm:effective NP}
	Suppose $X = (X_1,\dots, X_d)\in \cM_n^d$ is a row contraction and $Y = (Y_{i,j})_{i,j=1}^{s,t}\in \cM_{s\times t}\otimes \cM_n$
	is an $s\times t$ block matrix with $n\times n$ blocks.
	There exists an nc function $f\in \cM_{s\times t}\otimes H^\infty(\B^d)$ such that $\norm{f} \leq 1$ and $f(X) = Y$ if and only if 
	\[
		P_X\otimes I_s - (I_n\otimes Y)(P_X\otimes I_t)(I_n\otimes Y^*) \succeq 0.
	\]
\end{theorem}

\begin{proof}
	Let $\Phi:\cM_n\to \cM_{ns}$ be the operator defined by
	\[
		\Phi(H) = \sum_{w\in \fax} (X^w H X^{w*})\otimes I_s - Y\left(\sum_{w\in \fax}(X^w H X^{w*})\otimes I_t\right)Y^*.
	\]
	Next observe
	\begin{align*}
		\sum_{i,j} E_{ij}\sot \sum_w X^w E_{ij} X^{w*}\sot I_s
			&= \sum_w \sum_{i,j} (E_{ij}\sot I)(I\sot X^wE_{ij}X^{w*})\sot I_s\\
			&= \sum_w \sum_{i,j} (I\sot X^w)(E_{ij}\sot E_{ij})(I\sot X^{w*})\sot I_s\\
			&= \sum_w (I\sot X)^w \Choi_n (I\sot X)^{w*}\sot I_s\\
			&= P_X\otimes I_s,
	\end{align*}
	where the last equality uses Lemma~\ref{lem:Choi to P}.
	Hence,
	\begin{align*}
		(I_n\sot \Phi)(\Choi_n) &= P_X\otimes I_s - (I_n\otimes Y)(P_X\otimes I_t)(I_n\otimes Y^*).
	\end{align*}
	Thus, Choi's Theorem tells us $\Phi$ is completely positive if and only if 
		$P_X\sot I_s - (I_n\sot Y)(P_X\sot I_t)(I_n\sot Y^*)\succeq 0$.
	Finally, as already noted, \cite[Theorem 6.5]{BMV18} says that $\Phi$ is completely positive if and only if there is a solution to the interpolation problem. This completes the proof.
\end{proof}



We now turn an essentially equivalent version of the interpolation problem: if $X$ is a row contraction and $Y$ is given, find the minimal 
norm of a solution $f$ to the interpolation problem $f(X)=Y$. First of all, we must note that there may not be any $f$ with $f(X)=Y$; this 
will happen if and only if the blocks of $Y$ belong to the subalgebra generated by $X_1, \dots, X_d$. By the main Theorem 
of \cite{pascoe-2019-alg}, we know that a matrix $Z$ is in the algebra generated by $X$ if and only if $\vecc(Z)\in \ran(P_X)$. When each 
block of $Y$ belongs to this algebra, then there will exist an nc polynomial matrix $f$ with $f(X)=Y$.

We will define the NP-norm of $Y$ to be the minimal norm of a solution to $f(X)=Y$, and show how to compute this minimal norm using $P_X$. 
\begin{definition}
	Suppose $X = (X_1,\dots, X_d)\in\B^d\subset \cM_n^d$ and $Y\in \cM_{s\times t}\otimes \cM_n$.
	We define the $\NP{X}$ norm of $Y$ to be
	\begin{equation}\label{NP_norm_def}
		\norm{Y}_{\NP{X}} := \inf_{\substack{f\in M_{s\times t}\otimes H^\infty(\mathcal B^d) \\ f(X) = Y}} \norm{f}_{H^\infty}
	\end{equation}
	and note that implicitly we consider only nc functions $f$.
	Moreover, if $\|\sum_{i=1}^d X_iX_i^*\|=1$, so that $X$ lies in the boundary of $\mathcal B^d$ at level $n$, we define the $\ANP{X}$ norm of $Y$  (the asymptotic $\NP{X}$ norm at the boundary point $X$)  as
	\begin{equation}\label{ANP_norm_def}
		\|Y\|_{\ANP{X}}:=\lim_{t \nearrow1}\norm{Y}_{\NP{tX}}.
	\end{equation}
\end{definition}
We note that the $\NP{X}$ norm could be equivalently defined by taking the infimum just over nc polynomial matrices $f$ in the expression 
(\ref{NP_norm_def})

We will compute $\snorm{Y}_{\NP{X}}$ for all $Y$ whose blocks are in the algebra generated by $X$, and $\snorm{Y}_{\ANP{X}}$ for a special 
subclass of boundary points $X$, which will be useful in applications.  We first introduce some notation and make some elementary 
observations.

\begin{definition}
	If $X = (X_1,\dots, X_d)\in \cM_n^d$ then let $\alg_X$ denote the unital subalgebra of $\cM_n$ generated by $X_1,\dots, X_d$.
\end{definition}

Note:
\begin{itemize}
\item If $X = (X_1,\dots, X_d)\in \cM_n^d$ is a row contraction, then $P_X$ is self-adjoint and positive semi-definite,
so $P_X^{1/2}$ exists, and we let $\finv{P_X}$ denote the Moore-Penrose pseudoinverse of $P_X^{1/2}$.
\item Note $\finv{P_X}P_X^{1/2} = P_X^{1/2}\finv{P_X} = Q_X$, where $Q_X$ is the projection onto $\ran(P_X) = \vecc(\alg_X)$.
\item The matrices $P_X$ and $P_X^{1/2}$ are invertible if and only if $\alg_X = \cM_n$. In this case, $\finv{P_X} = P_X^{-1/2}$.	
\end{itemize}
%


\begin{corollary}
	\label{cor:NP norm formula}
	Suppose $Y\in \cM_{s\times t}\otimes \alg_X $. 
	If 
	\[
		{}^PY^P=(\finv{P_X}\otimes I_s)(I_n\otimes Y)(P_X^{1/2}\otimes I_t)
	\]
	then
	\[
		\norm{Y}_{\NP{X}} = \norm{{}^PY^P}.
	\]
\end{corollary}

\begin{proof}
	We begin by multiplying the main equation in Theorem~\ref{thm:effective NP} on the left and right by $(\finv{P_X}\otimes I_s)$:
	\begin{align*}
		Q_X\sot I_s
			&\succeq (\finv{P_X}\sot I_s)(I_n\sot Y)(P_X\sot I_t)(I_n\sot Y^*)(\finv{P_X}\sot I_s)\\
			&= {}^PY^P({}^PY^P)^*
	\end{align*}
	Suppose $c>0$. By considering the interpolation problem for $c^{-1}Y$ instead of $Y$, it follows from above that there exists 
		$f\in \cM_{s\times t}\sot H^\infty(\B^d)$ such that $\norm{f} = c$ and $f(X) = Y$ if and only if ${}^PY^P({}^PY^P)^*\preceq 
		c^2Q_X\sot I_s$ if and only if $\norm{{}^PY^P} \leq c$.
	If there exists $f\in \cM_{s\times t}\sot H^\infty(\B^d)$ such that $\norm{f} = c$ and $f(X) = Y$, then $\snorm{Y}_{\NP{X}}\leq 
	c$.
	Hence, $\snorm{{}^PY^P}\leq c$ implies $\snorm{Y}_{\NP{X}}\leq c$.
	
	On the other hand, if $\snorm{Y}_{\NP{X}} \leq c$ then for each $\ep>0$ there exists $f_\ep\in \cM_{s\times t}\sot H^\infty(\B^d)$ such 
	that $\norm{f_\ep} = c+\ep$ and $f_\ep(X) = Y$.
	However, this implies $\snorm{{}^PY^P}\leq c+\ep$ for all $\ep>0$, hence $\snorm{{}^PY^P}\leq c$.
	Thus, $\snorm{{}^PY^P} \leq c$ if and only if $\snorm{Y}_{\NP{X}}\leq c$.
	Therefore, $\snorm{Y}_{\NP{X}} = \snorm{{}^PY^P}$.
\end{proof}

\begin{theorem}\label{thm:ANPMainTheorem}
	Let $X = (X_1, \ldots, X_d)\in M_n(\mathbb{C})^d$ be a row co-isometry (that is, $\sum_{i=1}^d X_iX_i^*=I$). 
	If the algebra generated by $X_1, \ldots, X_d$ is all of $M_n(\mathbb{C})$ then
	\begin{align*}
		\lim_{t\to 1} (P_{tX}^{-1/2}\sot I_s)(I_n\sot Y)(P_{tX}^{1/2}\sot I_t)
			&= I_n\sot Y
	\end{align*}
	and consequently
	\[
		\norm{Y}_{\ANP{X}} = \lim_{t\to 1} \norm{Y}_{\NP{tX}} = \norm{Y}.
	\]
\end{theorem}

\begin{proof}
	We claim $\frac{1-t^2}{t^2}P_{tX} \rightarrow \overline{W} \otimes I_n$ as $t\rightarrow 1$, for some positive definite matrix $W \in M_n(\mathbb{C})$. If the claim is true, then 
	\begin{align*}
		\lim_{t\to 1} (P_{tX}^{-1/2}\sot I_s)(I_n\sot Y)(P_{tX}^{1/2}\sot I_t)
			&= (\overline{W}^{-1/2}\sot I_n\sot I_s)(I_n\sot Y)(\overline{W}^{1/2}\sot I_n\sot I_t)\\ 
			&= I_n\sot Y
	\end{align*}
from which we conclude, by Corollary~\ref{cor:NP norm formula}, that $\|Y\|_{\ANP{X}} =\|Y\|$. Thus, it is sufficient to prove the claim.

	Consider $T = \sum_{i=1}^d \overline{X_i} \otimes X_i.$ Using the identity (\ref{eqn:classical_vec_identity}) we see 
\[
T \vecc (I) = \vecc \left(\sum_{i=1}^d X_iX_i^*\right) =\vecc(I), 
\]
that is, $\vecc(I)$ is an eigenvector for $T$ with eigenvalue $1.$ Since $X$ is a row contraction and the algebra generated by $X_1, 
\ldots, X_d$ is all of $M_n(\mathbb{C}),$ it follows from the quantum Perron-Frobenius theorem of Evans and H\o egh-Krohn 
\cite{evans-hoegh-krohn-1978} that the spectral radius of $T$ is equal to $1$, and the (generalized) eigenspace corresponding to $1$ is one 
dimensional.  (For a treatment of this result more tailored to the present application, see \cite[Theorem 5.4]{pascoe-2019-QPF}.)
	Let $\vecc(W)$ be the corresponding left eigenvector to $1$ (that is, $\vecc(W)^*T=\vecc(W)^*$), normalized so that $\vecc(W)^*\vecc(I) = 
	1$. From \cite[Theorem 5.4]{pascoe-2019-QPF} and the remarks following it, the matrix $W$ must be positive definite. (This conclusion 
	again relies on the fact that $X_1, \dots , X_d$ generate all of $\cM_n$.) 
	Thus, we may decompose $T$ as
	\[
		T = G+B
	\]
	where $G= \vecc (I) \vecc (W)^*,$ and $B$ is the remainder $T-G.$  It follows that $B$ has spectral radius less than or equal to $1$, and that $1$ is not an eigenvalue of $B$.
	Next, we see that
	\[
		G^2 = \vecc(I)\vecc(W)^*\vecc(I)\vecc(W)^* = \vecc(I)\vecc(W)^* = G.
	\]
	Moreover, since $\vecc(I)$ is a right eigenvector to $1$ and $\vecc(W)$ is a left eigenvector to $1$, we also have
	\[
		TG = G = GT
	\]
	and consequently $GB = BG = 0$.
	In particular, $T^n = (G+B)^n = G^n + B^n = G + B^n$ and
	\begin{align*}
		(I - t^2 T)^{-1} 
			&= \sum_{n=0}^\infty (t^2 T)^n = I + \sum_{n=1}^\infty (t^{2n}G + t^{2n}B^n)\\
			&= I + \frac{t^2}{1-t^2}G + t^2B(I - t^2B)^{-1}.
	\end{align*}
	
	Thus,
	\[
		P_{tX} = [(1-t^2T)^{-1}]^\psi = \left[1 + \frac{t^2}{1-t^2}G + t^2B(1-t^2B)^{-1}\right]^\psi. 
	\]
	Since, as noted above, $r(B)\leq 1$ and $1$ is not an eigenvalue of $B$, we have
	\[
	\lim_{t\to 1} (1-t^2) B(1-t^2 B)^{-1}=0.
	\]
	Taking the limit of $\frac{1-t^2}{t^2}P_{tX}$ as $t\rightarrow 1,$ we obtain
	\[
		\lim_{t\rightarrow 1} \frac{1-t^2}{t^2}P_{tX} =G^\psi= [\vecc I (\vecc W)^*]^{\psi}=\overline W \otimes I_n,
	\]
	which finishes the proof of the claim, and hence the theorem. 
\end{proof}
We restate the above theorem in terms of the condition of the interpolation problem in Section
\ref{sec:remarks on NPX norm}.

\section{Boomerang Matrix}
In this section we collect some calculations which will be useful in the next section. 
\label{sec:boomerang matrix}
\begin{definition}
	Define $\breve{B}\in \cM_{n^3\times n}$ to be
	\[
		\breve{B} = \sum_{i,j=1}^n \vecc(E_{ij})\otimes E_{ij}
			= \sum_{i,j=1}^n e_i\otimes e_j\otimes E_{ij}
	\]
	The matrix $\breve{B}$ is known as the {\bf Boomerang matrix}.
\end{definition}

\begin{lemma}
	If $C\in \cM_n$, then
	\begin{equation}
	\label{eq:Boomerang switch 1}
		(C\otimes I_n\otimes I_n)\breve{B} = (I_n\otimes I_n\otimes C^T)\breve{B}.
	\end{equation}
	Moreover, if $A\in \cM_{n^2}$ and $D\in \cM_n$ then 
	\begin{equation}
	\label{eq:Boomerang switch 2}
		\breve{B}^T(A\otimes CD)\breve{B} = \breve{B}^T([(C^T\otimes I)A(D^T\otimes I)]\otimes I)\breve{B}.
	\end{equation}
\end{lemma}

\begin{proof}
	We prove the first item for $C=E_{k\ell}$ and then extend linearly to all of $\cM_n$.
	Note
	\begin{align*}
		(E_{k\ell}\otimes I\otimes I)\breve{B} 
			&= \sum_{i,j}^n(E_{k\ell}\otimes I\otimes I)(e_i\otimes e_j\otimes E_{ij})
				= \sum_{j}^n e_k\otimes e_j \otimes E_{\ell j}\\
			&= \sum_{i,j}^n e_i\otimes e_j\otimes E_{\ell k}E_{ij}
				= \sum_{i,j}^n(I\otimes I\otimes E_{\ell k})(e_i\otimes e_j\otimes E_{ij})\\
			&= (I\otimes I\otimes E_{\ell k})\breve{B}.
	\end{align*}
	Extending linearly we have Equation \eqref{eq:Boomerang switch 1}.
	
	Now suppose $A\in \cM_{n^2}$ and $D\in \cM_n$.
	Observe
	\begin{align*}
		(A\otimes D)\breve{B} 
			&= (A\otimes I)(I\otimes I\otimes D)\breve{B} = (A\otimes I)(D^T\otimes I \otimes I)\breve{B}\\
			&= ([A (D^T\otimes I)]\otimes I)\breve{B},
	\end{align*}
	and by taking transposes we have
	\[
		\breve{B}^T(A\otimes C) = \breve{B}^T([(C^T\otimes I)A]\otimes I).
	\]
	Using \eqref{eq:Boomerang switch 1} finally yields
	\begin{align*}
		\breve{B}^T(A\otimes CD)\breve{B} &= \breve{B}^T (A\otimes C)(I\otimes I\otimes D)\breve{B}\\
		&= \breve{B}^T([(C^T\otimes I)A(D^T\otimes I)]\otimes I)\breve{B}.
	\end{align*}
\end{proof}

\begin{lemma}
	\label{lem:P sum X to Choi}
	For any row contraction $X\in \cM_n^d$, 
	\[
		P_X - \sum_{i=1}^d (X_i^T\otimes I_n)P_X(\bar{X_i}\otimes I_n) = [I_{n^2}]^\psi = \mathfrak{C}_n,
	\]
	where $\mathfrak{C}_n = \sum_{i,j=1}^n E_{ij}\otimes E_{ij}$ is the Choi matrix.
\end{lemma}

\begin{proof}
	Observe by Proposition \ref{prop:modularity of psi},
	\begin{align*}
		P_X - {\textstyle\sum}_{i=1}^d(X_i^T\sot I_n)P_X(\bar{X_i}\sot I_n) 
			&= P_X - \left[{\textstyle\sum}_{i=1}^d(I\sot I)P_X^\psi(\bar{X_i}\sot X_i)\right]^\psi\\
			&= \Big[(P_X)^\psi\Big]^\psi - \left[P_X^{\psi}{\textstyle\sum}_{i=1}^d(\bar{X_i}\sot X_i)\right]^\psi\\
			&= \left[[P_X]^\psi(I - {\textstyle\sum}_{i=1}^d \bar{X_i}\sot X_i)\right]^\psi\\
			&= \left[(I - {\textstyle\sum}_{i=1}^d \bar{X_i}\sot X_i)^{-1}(I - {\textstyle\sum}_{i=1}^d \bar{X_i}\sot X_i)\right]^\psi\\
			&= [I_{n^2}]^\psi = \mathfrak{C}_{n}.
	\end{align*}
\end{proof}

\begin{lemma}
	\label{lem:Choi kron prod XHX}
	Suppose $X\in \cM_n^d$ is a row contraction.
	If $H\in \cM_n$ then
	\[
		\breve{B}^T\left(P_X \otimes \left(H - {\textstyle\sum}_{i=1}^d X_iHX_i^*\right)\right)\breve{B} = H.
	\]
\end{lemma}

\begin{proof}
	We begin by computing the left hand side for a fixed $i$:
	\begin{align*}
		\breve{B}^T(P_X \sot (H - X_iHX_i^*))\breve{B} 
			&= \breve{B}^T(P_X \sot H)\breve{B} - \breve{B}^T(P_X \sot  X_iHX_i^*)\breve{B}\\
			&= \breve{B}^T(P_X \sot H)\breve{B} - \breve{B}^T([(X_i^T\sot I)P_X(\bar{X_i}\sot I)] \sot  H)\breve{B}\\
			&= \breve{B}^T([P_X - (X_i^T\sot I)P_X(\bar{X_i}\sot I)]\sot H)\breve{B}.
	\end{align*}
	Summing over $i$ and applying Lemma~\ref{lem:P sum X to Choi} implies
	\[
		\breve{B}^T\left(P_X \otimes \left(H - {\textstyle\sum}_{i=1}^d X_iHX_i^*\right)\right)\breve{B} 
			= \breve{B}^T(\mathfrak{C}_n\otimes H)\breve{B}.
	\]
	Using the definition of the Choi matrix we finish the computation:
	\begin{align*}
		\breve{B}^T(\mathfrak{C}_n\otimes H)\breve{B} 
			&= \sum (e_{i_1}^T\otimes e_{j_1}^T\otimes E_{j_1 i_1})(\mathfrak{C}_n\otimes H)(e_{i_2}T\otimes e_{j_2}\otimes E_{i_2 j_2})\\
			&= \sum(e_{i_1}^T\otimes e_{j_1}^T)\mathfrak{C}_n(e_{i_2}\otimes e_{j_2}) \otimes E_{j_1i_1}HE_{i_2j_2}\\
			&= \sum e_{i_1}^TE_{k\ell}e_{i_2}\otimes e_{j_1}^TE_{k\ell}e_{j_2}\otimes E_{j_1i_1}HE_{i_2j_2}\\
			&= \sum_{k,\ell} 1\otimes 1\otimes E_{k k}HE_{\ell \ell}\\
			&= H.
	\end{align*}
\end{proof}

\begin{definition}
	Let $s,t\in\Z_+$ and let $\cQ_{n,s}\in \cM_{ns}$ denote the permutation matrix such that
	\[
		\cQ_{n,s}(U\otimes V)\cQ_{n,s}^T = V\otimes U
	\]
	for all $U\in \cM_n$ and $V\in \cM_s$.
	In particular, if $W\in \cM_{n\times m}$ and $Z\in \cM_{t\times s}$ then
	\[
		\cQ_{n,t}(W\otimes Z)\cQ_{m,s}^T = Z\otimes W.
	\]

	For each $r\geq 1$, we define the $n^3r\times nr$ matrix
	\[
		\breve{B}_r = \left(\sum_{i,j} e_i\otimes  e_j \otimes I_r\otimes E_{ij}\right)\cQ_{n,r}
	\]
	to be the {\bf ampliated boomerang matrix}.
\end{definition}

\begin{proposition}
	\label{prop:ampliated props}
	Suppose $A\in \cM_{n^2}$, $C\in \cM_n$, $Z\in \cM_{nt\times ns}$ and $W\in \cM_{ns\times nt}$.
	We have the following identities;
	\begin{equation}
		\label{eq:right amp boom eq}
		\big[(A\otimes I_t)(I_n\otimes Z)\otimes C\big]\breve{B}_s = \big[A\otimes I_t\otimes C\big]\breve{B}_t Z
	\end{equation}
	and
	\begin{equation}
		\label{eq:left amp boom eq}
		\breve{B}_s^T\big[(I_n\otimes W)(A\otimes I_t)\otimes C\big] = W\breve{B}_t^T\big[A\otimes I_t\otimes C\big].
	\end{equation}
	If, in addition, $J,K\in \cM_{ns}$, then
	\begin{equation}
		\label{eq:double amp boom eq}
		\breve{B}_s^T\big[(I_n\sot J)(A\sot I_s)(I_n\sot K) \sot C\big]\breve{B}_s = J\breve{B}_s^T\big[A\otimes I_s\otimes C\big]\breve{B}_sK.
	\end{equation}	
	
	Furthermore, the ampliated boomerang matrix satisfies the ampliated versions of Equations~\eqref{eq:Boomerang switch 1} and \eqref{eq:Boomerang switch 2}.
\end{proposition}

\begin{proof}
	Suppose $A\in \cM_{n^2}$ and $C\in \cM_n$.
	Let $\cE_{k\ell}$ be the $t\times s$ matrix with a 1 in the $k,\ell$-entry and zeros elsewhere.
	We prove Equation~\eqref{eq:right amp boom eq} with $E_{pq}\otimes \cE_{k\ell}$ first:
	\begin{align*}
		[(A\sot I_t)(I_n\sot E_{pq}\sot \cE_{k\ell})\sot C]\breve{B}_s
			&= {\textstyle\sum}_{ij} \big[A(I_n\sot E_{pq})\sot \cE_{k\ell}\sot C\big]
				\big[e_i\sot e_j\sot I_s\sot E_{ij}\big]\cQ_{n,s}\\
			&= {\textstyle\sum}_{ij} [A(e_i\sot E_{pq}e_j)\sot \cE_{k\ell}\sot CE_{ij}]\cQ_{n,s}\\
			&= {\textstyle\sum}_{i}  [A(e_i\sot e_p)\sot \cE_{k\ell}\sot CE_{iq}]\cQ_{n,s}\\
			&= {\textstyle\sum}_{ij} [A(e_i\sot e_j)\sot \cE_{k\ell}\sot CE_{ij}E_{pq}]\cQ_{n,s}\\
			&= {\textstyle\sum}_{ij}(A\sot I_t\sot C)\left(e_i\sot  e_j \sot I_t\sot E_{ij}\right)(\cE_{k\ell}\sot E_{pq})\cQ_{n,s}\\
			&= (A\sot I_t\sot C)(\breve{B}_t\cQ_{n,t}^T)(\cE_{k\ell}\sot E_{pq})\cQ_{n,s}\\
			&= (A\sot I_t\sot C)(\breve{B}_t)(E_{pq}\sot \cE_{k\ell}).
	\end{align*}
	Thus, with linearity and by taking adjoints we have Equations~\eqref{eq:right amp boom eq} and \eqref{eq:left amp boom eq}.
	
	For Equation \eqref{eq:double amp boom eq}, set $t = s$ and combine Equations \eqref{eq:right amp boom eq} and \eqref{eq:left amp boom eq}:
	\begin{align*}
		\breve{B}_s^T\big[A\otimes JK\otimes C\big]\breve{B}_s 
			&= J\breve{B}_s^T[A\otimes I_s \otimes C]\breve{B}_s K.
	\end{align*}
	
	Finally, the ampliated versions of Equations~\eqref{eq:Boomerang switch 1} and \eqref{eq:Boomerang switch 2}
	follow readily from adapting their proofs.
\end{proof}

\begin{proposition}
	\label{prop:amp PX - XHX}
	Suppose $X\in \cM_n^d$ is a row contraction and $H\in \cM_n$.
	An ampliated version of Lemma~\ref{lem:Choi kron prod XHX} is satisfied:
	\[
		\breve{B}_s^T\left[P_X \otimes I_s\otimes \left(H - {\textstyle\sum}_{i=1}^d X_iHX_i^*\right)\right]\breve{B}_s = H\otimes I_s.
	\]
\end{proposition}

\begin{proof}
	This follows from adapting the proof of Lemma~\ref{lem:Choi kron prod XHX}.
\end{proof}

\section{Popescu Mini-Dilations}
\label{sec:mini-dilations}
Once more, we suppose $X = (X_1,\dots, X_d)\in \cM_n^d$ is a row-contraction.
Set $\Delta_X = I_n - \sum_{i=1}^d X_iX_i^*$ and observe that both $P_X$ and $\Delta_X$ are self-adjoint and positive semi-definite.

Define $\mathcal{V}_X = (P_X^{1/2}\sot I_n \sot \Delta_X^{1/2})\breve{B}_n \in \cM_{n^4\times n^2}$ and remark that 
	Proposition~\ref{prop:amp PX - XHX} with $H = I_n$ implies $\cV_X$ is an isometry: $\cV_X^*\cV_X = I_{n^2}$.
	
Recall $\finv{P_X}$ is the pseudoinverse of $P_X^{1/2}$ and $Q_X$ is the projection onto $\vecc(\alg_X)$, where $\alg_X$ is the unital algebra generated by $X_1,\dots, X_d$.

%
%

\begin{lemma}
	\label{lem:alg projection}
	If $W\in \alg_X$ then
	\[
		Q_X(I_n\sot W)P_X = (I_n\sot W)P_X \quad \text{ and } \quad P_X(I_n\sot W^*)Q_X = P_X(I_n\sot W^*).
	\]	
\end{lemma}

\begin{proof}
	We begin by taking $v\in \cM_{n^2}$ and recalling that $\ran(P_X) = \vecc(\alg_X)$, hence $P_X v = \vecc(V)$, for some $V\in \alg_X$.
	Moreover, since $W$ is also in $\alg_X$, it follows that $WV\in \alg_X$ and $Q_X\vecc(WV) = \vecc(WV)$.
	Thus,
	\begin{align*}
		Q_X(I_n \sot W)P_X vsec:failure of column-row
			&= Q_X(I_n\sot W)\vecc(V) = Q_X\vecc(WV)\\
			&= \vecc(WV) = (I_n\sot W)\vecc(V)\\
			& = (I_n\sot W)P_X v,
	\end{align*}
	allowing us to conclude that $Q_X(I_n \sot W)P_X = (I_n \sot W)P_X$.
	Taking adjoint shows
	\[
		P_X(I_n\sot W^*)Q_X = P_X(I_n\sot W^*).
	\]
\end{proof}

\begin{theorem}
	\label{thm:mini-dilation}
	Suppose $\alpha,\beta\in \C\fralg{x}$.
	If $\tilde{X} = \finv{P_X}(I_n\otimes X)P_X^{1/2}$ then
	\[
		\cV_X^*\big(\alpha(\tilde{X})\beta(\tilde{X})^*\otimes I_{n^2}\big)\cV_X = \alpha(X)\beta(X)^*\otimes I_n.
	\]
\end{theorem}

\begin{proof}
	Observe that Lemma~\ref{lem:alg projection} implies that $\alpha(\tilde{X}) = \finv{P_X}(I_n\otimes \alpha(X))P_X^{1/2}$.
	Take $W,Z\in \alg_X$ and set 
	\[
		\tilde{W} = \finv{P_X}(I_n\otimes W)P_X^{1/2} \quad \text{ and } \quad \tilde{Z} = \finv{P_X}(I_n\otimes Z)P_X^{1/2}.
	\]
	Note $\tilde{Z}^* = P_X^{1/2}(I_n\sot Z^*)\finv{P_X}$ and applying Lemma~\ref{lem:alg projection} once more implies
	\begin{align*}
		P_X^{1/2}\tilde{W}\tilde{Z}^*P_X^{1/2} 
			&= P_X^{1/2}\finv{P_X}(I_n\sot W)P_X^{1/2}P_X^{1/2}(I_n\sot Z^*)\finv{P_X}P_X^{1/2}\\
			&= Q_X(I_n\sot W)P_X(I_n \sot Z^*)Q_X \\
			&= (I_n \sot W)P_X (I_n \sot Z^*).
	\end{align*}
	Using Proposition~\ref{prop:ampliated props} we have the following chain of equalities:
	\begin{align*}
		\cV_X^*(\tilde{W}\tilde{Z}^*\sot I_{n^2})\cV_X 
			&= \left(\left(P_X^{1/2}\sot I_n \sot \Delta_X^{1/2}\right)\breve{B}_n\right)^*(\tilde{W}\tilde{Z}^*\sot I_{n^2})
				\left(\left(P_X^{1/2}\sot I_n \sot \Delta_X^{1/2}\right)\breve{B}_n\right)\\
			&= \breve{B}_n^T\left(P_X^{1/2}\tilde{W}\tilde{Z}^*P_X^{1/2}\sot I_n\sot \Delta_X\right)\breve{B}_n\\
			&= \breve{B}_n^T\big((I_n\sot W)P_X(I_n\sot Z^*)\sot I_n \sot \Delta_X\big)\breve{B}_n\\
			&= \breve{B}_n^T\big((I_n\sot W\sot I_n)(P_X\sot I_n)(I_n\sot Z^*\sot I_n)\sot \Delta_X\big)\breve{B}_n \\
			&= (W\sot I_n)\breve{B}_n^T\left(P_X\sot I_n\sot \Delta_X\right)\breve{B}_n(Z^*\sot I_n) \\
			&= WZ^*\sot I_n.
	\end{align*}
	Since $\alpha(X),\beta(X)\in \alg_X$, setting $W = \alpha(X)$ and $Z = \beta(X)$ finishes the proof.

\end{proof}

%
%
%

\section{Proof of Theorem \ref{thm:row-col fails for Fock}}
\label{sec:failure of column-row}

\begin{proof}
It suffices to prove the theorem in the case $d=2$. Fix $n$ and choose a pair of $n\times n$ matrices $X=(X_1, X_2)$ such that 
$X_1X_1^*+X_2X_2^* =I_n$ and $X_1, X_2$ generate all of $\cM_n$ as a unital algebra. (A construction of such a pair is given in the next 
section, see Example \ref{group_examples_section}(a).) Consider the $n\times n$ matrices
\[
Y_1 =E_{11}, \quad Y_2=E_{12}, \quad \dots, \quad Y_n=E_{1n}.
\]
Put
\[
Y_{col} =\begin{bmatrix} Y_1 \\ Y_2 \\ \vdots \\ Y_n\end{bmatrix}, \quad Y_{row} =\begin{bmatrix} Y_1 & Y_2 & \cdots & Y_n\end{bmatrix},
\]
then $\|Y_{col}\|=\|\sum_{i=1}^n Y_i^*Y_i\|^{1/2}=1$ and $\|Y_{row}\|=\|\sum_{i=1}^n Y_iY_i^*\|^{1/2}=\sqrt{n}$. 

Let $0<\epsilon<1$. By Theorem~\ref{thm:ANPMainTheorem}, for all $t$ sufficiently close to $1$ we have both
\[
\|Y_{col}\|_{\NP{tX}} <(1+\epsilon) \quad \text{and} \quad \|Y_{row}\|_{\NP{tX}} >(1-\epsilon) \sqrt{n}.
\]
Fix such a $t$. By the definition of the $NP(tX)$ norm, there exists an $n\times 1$ column of elements of $\mathcal L_2$
\[
F_{col}=\begin{bmatrix} f_1 \\ f_2 \\ \vdots \\ f_n\end{bmatrix}
\]
such that $\|F_{col}\|_\infty <1+\epsilon$ and $F_{col}(tX)=Y_{col}$, that is, $f_i(tX)=Y_i$ for each $i=1, \dots ,n$, and
 \begin{equation}\label{col-row-proof-display-1}
 \|\sum_{i=1}^n f_i^*f_i\|^{1/2}<1+\epsilon.  
 \end{equation}
 If we take these $f_1, \dots , f_n$ and form the row
\[
F_{row} =\begin{bmatrix} f_1 & f_2 & \cdots & f_n\end{bmatrix}
\]
then $F_{row}$ solves the interpolation problem $F_{row}(tX)=Y_{row}$, and hence, again by the definition of the $NP(tX)$ norm, we must have 
\begin{equation}\label{col-row-proof-display-2}
\|\sum_{i=1}^n f_i f_i^*\|^{1/2} =\|F_{row}\|_\infty \geq  \|Y_{row}\|_{\NP{tX}} >(1-\epsilon) \sqrt{n}.
\end{equation}
Comparing (\ref{col-row-proof-display-1}) and (\ref{col-row-proof-display-2}), and keeping in mind that $\epsilon$ was arbitrary, we conclude 
that $C_n\geq \sqrt{n}$.  As noted earlier, the reverse inequality always holds, so the theorem is proved. 

\end{proof}

%
%
%

\section{Examples}\label{sec:examples}

As we have seen, a central role is played by $d$-tuples of $n\times n$ matrices $X=(X_1, \dots, X_d)$ with the following two properties:
\begin{itemize}
\item the row $X$ is a {\em co-isometry}, i.e. $\sum_{i=1}^d X_iX_i^* =I_n$, and
\item $X$ is {\em irreducible} in the sense that $\alg_X=\cM_n$. 
\end{itemize}

We now give several examples of such systems $X$; the first is important for the proof of Theorem~\ref{thm:ANPMainTheorem} in the sense that 
it shows that such systems exist for $d=2$ and all $n$ (hence for all $d$ and $n$). 
\subsection{Irreducible representations of groups}\label{group_examples_section}
\begin{enumerate}
\item[a)] Let $d=2$ and let $X = (\frac{1}{\sqrt{2}}S, \frac{1}{\sqrt{2}}M)$ where $S$ is the cyclic permutation matrix and $M$ is the discrete Fourier transform of $S$. That is, 
	\[
	Se_{i} = e_{i+1 \pmod{n}}, Me_i = \omega^{i} e_i
	\] 
where $\omega$ is an n-th root of unity. Since $S$ and $M$ are unitary, it is trivial that $X$ is a row co-isometry. Again since $S$ and $M$ 
are unitary, it follows that the algebra they generate is a $*$-algebra, and it is straightforward to check that the only matrices commuting 
with both $S$ and $M$ are scalar multiples of the identity. Thus, $\alg_X=\cM_n$. 
	\item[b)] More generally, given any group $G$ with generators $g_1, \ldots g_d,$ we can consider any irreducible unitary representation 
	$\pi:G\to \cM_n$ and let $X_i = w_i\pi(g_i)$, $i=1, \dots , d$ where the $w_i$ are nonzero and $\sum |w_i|^2=1.$ (Note that, in the 
	previous example, $M$ and $S$ generate a group of cardinality $n^3.$) As before the algebra generated by the $X_i$ is a $*$-algebra, and 
	hence the irreducibility of the representation implies that $\alg_X =\cM_n$. 
\end{enumerate}

\subsection{Many variable example: the Choi point} \label{choi_example_section} When $d=n^2$, we can construct a special $d$-tuple of 
$n\times n$ matrices $X_1, \dots , X_d$, for which it is easy to verify the conclusion of Theorem~\ref{thm:ANPMainTheorem} directly, without 
appeal to the machinery of the quantum Perron-Frobenius theorem. (In fact this is the context in which the failure of the column-row property 
for $\mathcal L_d$ was originally discovered. In particular, using the following lemma and imitating the proof of 
Theorem~\ref{thm:ANPMainTheorem}, one can show that for $\mathcal L_{n^2}$ the column-row constant $C_n$ is $\sqrt{n}$. Since all the 
$\mathcal L_d$ embed completely isometrically in $\mathcal L_2$, one concludes $C_n=\sqrt{n}$ in $\mathcal L_2$, for all $n$.)
\begin{lemma}\label{lem:choi_point_calculuation}
	Fix $n>1$ and let $d=n^2$.  We consider the $n^2$ matrices $X_{i,j}$, each of size $n\times n$,
\[	
	X _{i,j}= \frac{1}{\sqrt{n}}E_{ij}, \quad {1 \leq i,j\leq n} 
\]
arranged into a row (say, by listing the subscripts $(ij)$ in lexicographic order). Then	$X$ is a row co-isometry and 
\[
\lim_{t\nearrow 1}\frac{1-t^2}{t^2} P_{tX} =\frac{1}{n} I_{n^2}.
\]
Hence for any $Y\in \cM_{s\times t}\otimes \cM_{n^2}$,
\[
\|Y\|_{\ANP{X}} =\|Y\|.
\]
\end{lemma}

\begin{proof}
It is straightforward to verify that $X$ is a row co-isometry.
	Now, let 
\[	
	T = \sum_{i,j=1}^n \overline{X_{i,j}}\otimes X_{i,j} = \frac{1}{n}\sum_{i,j=1}^n E_{ij} \otimes E_{ij} =\frac{1}{n}\Choi_n. 
\]	
Note $T^2 = T.$
	Moreover, $T^{\psi} = \frac{1}{n}I_{n^2}.$
	So, computing $P_{tX},$ we see that
	\begin{align*}
		P_{tX} 	&= [(I-t^2T)^{-1}]^{\psi} \\
			&= [I+\frac{t^2}{1-t^2}T]^\psi\\
			&= nT+\frac{t^2}{1-t^2}\frac{1}{n}I.
	\end{align*}
This proves the first claim of the lemma, and the second claim follows exactly as in the proof of Theorem~\ref{thm:ANPMainTheorem}.
\end{proof}

%
%
%

\section{Further remarks on the $\NP{X}$ norm and interpolating sequences}
\label{sec:remarks on NPX norm}

The Nevanlinna-Pick norm of a block matrix $Y$ at $X,$ denoted $\|Y\|_{\NP{X}},$ is the minimum block $H^\infty$ norm of a function $f$ satisfying the equation $f(X)=Y.$
We define the {\bf condition number} of $X,$ denoted $\kappa(X),$ by the formula
\[
	\kappa(X) = \inf_{Y \neq 0, Y\in \cM_{s\times t}\otimes \alg_X} \frac{\|Y\|_{\NP{X}}}{\|Y\|}.
\]
Note that, by definition, for any $Y\in \alg_X$,
	$$\|Y\| \leq \|Y\|_{\NP{X}} \leq \kappa(X)\|Y\|.$$
The two following fairly harmless assertions, which will be established momentarily, have somewhat explosive consequences:
\begin{enumerate}
	\item If $X$ is an irreducible row co-isometry, then $\lim_{t\rightarrow 1} \kappa(tX) = 1.$
	\item If $X_1$ is a row contraction and $X_2$ is an irreducible row co-isometry, then $$\lim_{t\rightarrow 1} \kappa(X_1\oplus tX_2) = \kappa(X_1).$$
\end{enumerate}
Together, they will be used to establish the following fact, which is somewhat surprising in light of the failure of the column-row property for the free semigroup algebras: :  {\em For any sequence of contractive target data, there is an interpolating sequence for that data such that the interpolating function
	can be chosen with norm less than or equal to $1.$}


In fact, in this case, one can actually choose the interpolating sequence based only on the sequence of norms of the target data, and their sizes.

Recall the {\bf elementary Pick matrix}: 
	\[P_X= [(I-\sum \overline{X_i}\sot X_i)^{-1}]^\psi.\]
Given a set $S \subset \cM_n,$ we denote its {\bf commutant} by $S',$ and we denote the set of {\bf invertible elements} in $S$ by $S^\times.$
\begin{corollary}
	\label{cor:NP norm formula intro}
	Suppose $Y\in \cM_{s\times t}\otimes \alg_X$. 
	Suppose $D\in \{I \otimes X\}'^\times.$
	Let $Q_{X,D} = (DP_XD)^{1/2}$
	Then,
		\[\norm{Y}_{\NP{X}} = \norm{(Q_{X,D}^\dagger\otimes I_s)(I_n\otimes Y)(Q_{X,D}\otimes I_t)}.\]
	(Here, $\dagger$ denotes the Moore-Penrose pseudoinverse.) 
\end{corollary}
The corollary follows essentially trivially from Corollary~\ref{cor:NP norm formula} (by which it is sufficient to take $D=I$). However, especially in the case of multi-point Pick problems, $D$ can act as a pre-conditioner.
Note the necessity that $Y\in \cM_{s\times t}\otimes \alg_X,$ rather than merely $Y\in \cM_{s\times t}\otimes \cM_n.$

We define the {\bf effective condition number} of $X,$ denoted $\gamma(X),$ to be defined via the following formula,
\[
	\gamma(X) = \inf_{D \in (\{I \otimes X\}')^\times} \sqrt{\|DP_XD\|\|(DP_XD)^{-1}\|}.
\]
The effective condition number gives a bound on the condition number.
\begin{corollary}
  \label{cor:NP norm formula intro restate}
  For all $X\in \mathcal B^d$,
	\[\kappa(X)\leq\gamma(X).\]	
	
In particular, for all $Y\in \alg_X$ we have
        \[\norm{Y}_{\NP{X}} \leq \gamma(X)\norm{Y}.\]
\end{corollary}
We can also restate Theorem \ref{thm:ANPMainTheorem} in terms of condition numbers, which will be useful in the construction of interpolating sequences.
\begin{theorem}\label{ANPIntro} 
	Let $X = (X_1, \ldots, X_d)\in M_n(\mathbb{C})^d$ be a row co-isometry.
	If the algebra generated by $X_1, \ldots, X_d$ is all of $M_n(\mathbb{C})$ then
		$\lim \kappa(tX) = 1.$
	That is, for $Y\in \cM_{s\times t}\otimes \cM_n$ $\lim \|Y\|_{\NP{tX}} = \|Y\|.$
\end{theorem}

%
%
%

\subsection{Interpolating sequences}
\label{sec:interpolating sequences}

We now do some basic constructions of interpolating sequences.
\begin{lemma}
	If $X_1$ is a row contraction and $X_2$ is an irreducible row co-isometry, then the spectral radius of $T=\sum (\overline{X_1)_i} \otimes (X_2)_i$ is less than $1.$
\end{lemma}
\begin{proof}
	Note $T^n = \sum_{|w|=n} \overline{X_1}^w\otimes X_2^w.$
	Therefore, $(T^n)^\psi =  \sum_{|w|=n} \vecc X_1^w\otimes (\vecc X_2^w)^*.$
	So, $$\|(T^n)^\psi\|\leq \|\sup_{\sum_{|w|=n} |a_w|^2=1} a_wX_1^w \|\|\sup_{\sum_{|w|=n} |a_w|^2=1} a_wX_2^w \|.$$
	So, by the Gelfand formula for outer spectral radius, we see that the spectral radius of $T$ \cite{pascoe-2019-QPF} is less than the geometric mean of the outer spectral radii
	of $X_1$ and $X_2.$
\end{proof}

We now show that the condition number of a direct sum of some tuple with a scaled co-isometric tuple has the same condition number as the original in the limit.
\begin{lemma}
	If $X_1$ is a row contraction and $X_2$ is an irreducible row co-isometry, then $$\lim_{t\rightarrow 1} \kappa(X_1\oplus tX_2) = \kappa(X_1).$$
\end{lemma}
\begin{proof}
	The reader may verify that $P_{X_1\oplus tX_2}$ has a block $4$ by $4$ structure with four non-zero block entries, let $\hat{P}_{X_1\oplus tX_2}$ be the matrix with the
	zero columns and rows removed.
	Note,
		$$\hat{P}_{X_1\oplus tX_2} =
		\bbm  [(I-\sum \overline{(X_1)_i}\sot (X_1)_i)^{-1}]^\psi & [(I-t\sum \overline{(X_1)_i}\sot (X_2)_i)^{-1}]^\psi  \\
			 [(I-t\sum \overline{(X_2)_i}\sot (X_1)_i)^{-1}]^\psi & [(I-t^2\sum \overline{(X_2)_i}\sot (X_2)_i)^{-1}]^\psi   \ebm$$
	Preconditioning by a block diagonal $D$ with $1$ and $\sqrt{n(1-t^2)}$ on the diagonal, we get that 
		$$\tilde{P}_{X_1\oplus tX_2} =
		\bbm  [(I-\sum \overline{(X_1)_i}\sot (X_1)_i)^{-1}]^\psi & \sqrt{n(1-t^2)}[(I-t\sum \overline{(X_1)_i}\sot (X_2)_i)^{-1}]^\psi  \\
			 \sqrt{n(1-t^2)}[(I-t\sum \overline{(X_2)_i}\sot (X_1)_i)^{-1}]^\psi & {n(1-t^2)}[(I-t^2\sum \overline{(X_2)_i}\sot (X_2)_i)^{-1}]^\psi   \ebm$$
	Therefore, taking $t\rightarrow 1$	
		$$\lim_{t\rightarrow 1} \tilde{P}_{X_1\oplus tX_2} =
		\bbm  P_{X_1} & 0  \\
			 0 & I  \ebm$$
	Therefore, applying Corollary \ref{cor:NP norm formula intro}, $$\lim_{t\rightarrow 1} \kappa(X_1\oplus tX_2) = \kappa(X_1).$$
\end{proof}

We now immediately see the following theorem.
\begin{theorem}
	Given $(\rho_i)^\infty_{i=1},$ a sequence of numbers in $[0,1)$, and $(n_i)^\infty_{i=1},$ a sequence of natural numbers,
	there is an sequence $(X^{(i)})^\infty_{i=1}$
	such that each $X^{(i)}$ has size $n_i$ and for any sequence $(Y^{(i)})^\infty_{i=1}$ such that $\|Y^{(i)}\|\leq \rho_i,$ there is a
	function in $H^\infty$ of norm $1$ such that $f(X^{(i)})=Y^{(i)}.$ 
\end{theorem}

%
%
%

\section{Numerics and random examples}
\label{sec:numerics}
In this section we present a pseudocode version of what was used to initially find counter-examples to the column-row property for the Fock space.

\subsection{Code} The following pseudocode gives an algorithm that attempts to randomly generate tuples of matrices $X = (X_1,X_2)$ and
$Y = (Y_1,\dots, Y_m)$ that satisfy the argument in Theorem~\ref{thm:row-col fails for Fock}.
Much like the argument in Theorem~\ref{thm:row-col fails for Fock}, the algorithm presented relies on Corollary~\ref{cor:NP norm formula} and 
Theorem~\ref{thm:ANPMainTheorem}.

Recall that given a row contraction $X = (X_1,\dots, X_d)\in \B^d\subset \cM_n^d$ we can solve the interpolation to a block matrix $Y\in 
M_{s\times t}\otimes \cM_n$ if and only if $Y\in \cM_{s\times t}\otimes \alg_X$.
Thus our numeric approach to Theorem~\ref{thm:row-col fails for Fock} certainly requires at least that $X = (X_1, X_2)$ is a row contraction 
and $Y\in \cM_{1\times m}\otimes \alg_X$.
Recall that in this case Corollary~\ref{cor:NP norm formula} implies
\[
	\norm{{}^PY^P} = \norm{Y}_{\NP{X}}.
\]
Thus, we choose $Y_1,\dots, Y_m\in \alg_X$ and set
\[
	Y_\rowt = \bBm Y_1 & \dots & Y_m \eBm \quad \text{ and } \quad Y_\colt = \bBm Y_1 \\ \vdots \\ Y_m \eBm.
\]
The goal is to find choices of $Y_1,\dots, Y_m$ such that
\[
	\sqrt{m} \approx \frac{ \norm{{}^P(Y_\rowt)^P} }{ \norm{{}^P(Y_\colt)^P} } = \frac{ \norm{Y_\rowt}_{\NP{X}} }{ \norm{Y_\colt}_{\NP{X}} }.
\]

As was seen in Theorem~\ref{thm:row-col fails for Fock}, since the $\NP{X}$ norm is an infimum, there must be an interpolating function 
$F_\colt\in M_{m\times 1}(H^\infty(\B^d))$ such that $F_\colt(X) = Y_\colt$ and $\norm{F_\colt}_\infty \approx \norm{Y_\colt}_{\NP{X}}$.
Choosing $F_\rowt$ to be the row vector version of $F_\colt$, we have that $F_\rowt(X) = Y_\rowt$.
Since $\norm{Y_\rowt}_{\NP{X}}\leq \norm{F_\rowt}_\infty$ and $\norm{Y_\colt}_{\NP{X}} \approx \norm{F_\colt}_\infty$ we have the following
\[
	\frac{ \norm{Y_\rowt}_{\NP{X}} }{ \norm{Y_\colt}_{\NP{X}} } 
		\approx \frac{ \norm{Y_\rowt}_{\NP{X}} }{ \norm{F_\colt}_{\infty} }
		\leq \frac{ \norm{F_\rowt}_{\infty} }{ \norm{F_\colt}_{\infty} }
		\leq C_m \leq \sqrt{m}.
\]
Thus, with a correct choice of $Y$, we have that $\sqrt{m} \lesssim C_n \leq \sqrt{m}$.

Now, fix $n$ and $m$ and choose a cut-off value $\gamma <\sqrt{m}$.
The following pseudo-code describes a loop to find $X = (X_1, X_2)\in \cM_n^2$ and $Y_1,\dots, Y_m\in \cM_n$ that witness the ratio
$\snorm{Y_\rowt}_{\NP{X}}  > \gamma \snorm{Y_\colt}_{\NP{X}}$.
\begin{enumerate}[\bf 1: ]
	\item Set a cut off value $\gamma < \sqrt{m}$;
	\item Set the maximum ratio $M_r = 0$;
	\item Choose a sufficiently small $\ep>0$;
	\item {LOOP} while $M_r < \gamma$;
	\item Randomly generate $Z = (Z_1,Z_2)\in \cM_n^2$ such that $Z_1 Z_1^* + Z_2Z_2^*$ is invertible;
	\item Set $X = (1-\ep)(Z_1 Z_1^* + Z_2Z_2^*)^{-1/2}Z$;
	\item Compute $P_X = \big[(I_{n^2} - \overline{X_1}\otimes X_1 - \overline{X_2}\otimes X_2)^{-1} \big]^\psi$;
	\item Compute $P_X^{1/2}$ and $P_X^{\dagger/2}$;
	\item Select $v_1,\dots, v_m$ to be distinct eigenvectors of $P_X$ with the smallest positive associated eigenvalues;
	\item Set each $Y_i = \vecc^{-1}(v_i)$;
	\item Form $Y = (Y_1, \dots, Y_m)$;
	\item Compute $\snorm{{}^P(Y_\rowt)^P}$ and $\snorm{{}^P(Y_\colt)^P}$;
	\item IF $\gamma \snorm{{}^P(Y_\colt)^P} > \snorm{{}^P(Y_\rowt)^P}$;
	\item THEN set $M_r = \snorm{{}^P(Y_\rowt)^P}/\snorm{{}^P(Y_\colt)^P}$ and PRINT($X$, $Y$, $M_r$);
	\item ELSE set $M_r = \max\{M_r, \snorm{{}^P(Y_\rowt)^P}/\snorm{{}^P(Y_\colt)^P}\}$;
	\item END LOOP.
\end{enumerate}

The nature of the above algorithm implies that if $\ep$ is not sufficiently close to $0$, then typically the column-row ratio will not be 
close to $\sqrt{m}$ and the loop will never terminate.
It is perhaps advisable to randomly generate $\ep'\in (0,\ep)$ at each iteration of the loop to give a better chance that the loop terminates.
A functioning version of the above pseudo-code (including a method of computing $P_X$) can be found \href{https://urldefense.proofpoint.com/v2/url?u=https-3A__github.com_mericaugat_EffectiveNP-257D-257Bhere&d=DwIGAg&c=sJ6xIWYx-zLMB3EPkvcnVg&r=Ow9187WI0-zGf1nppcGtig&m=zjkqh4CWAlxo_VxIU1Xf76kxCEAnNbCvBddQMQ3AajM&s=Kp2g8aQ-ABGFOp5fXOUsq-7mjSvC80IOZxuHlwO86ZA&e= }.

\subsection{Committee spaces} 
We note that if we had not done the normalization to make $X$ asymptotically unitary in the above code and instead chose random tuples with independent entries, in the limit
we would not find examples with large column-row ratio, as was proven in \cite{pascoe-2019-committee}. That is, sequences of random multipliers usually satisfy the true column-row property. Originally, our group did not normalize this way, and only found examples with a ratio of about $1.0043$ after millions of trials.

\bibliographystyle{alpha}
\bibliography{EffectiveNP}

\begin{thebibliography}{AHMR19b}

\bibitem[AHMR19a]{AHMR-2019}
Alexandru Aleman, Michael Hartz, John~E. McCarthy, and Stefan Richter.
\newblock Interpolating {S}equences in {S}paces with the {C}omplete {P}ick
  {P}roperty.
\newblock {\em Int. Math. Res. Not. IMRN}, 2019(12):3832--3854, 2019.

\bibitem[AHMR19b]{AHMR-2018}
Alexandru Aleman, Michael Hartz, John~E. McCarthy, and Stefan Richter.
\newblock Weak products of complete {P}ick spaces.
\newblock {\em arXiv:1804.10693}, 2019.

\bibitem[BMV18]{BMV18}
Joseph~A. Ball, Gregory Marx, and Victor Vinnikov.
\newblock Interpolation and transfer-function realization for the
  noncommutative {S}chur-{A}gler class.
\newblock In {\em Operator theory in different settings and related
  applications}, volume 262 of {\em Oper. Theory Adv. Appl.}, pages 23--116.
  Birkh\"{a}user/Springer, Cham, 2018.

\bibitem[Dav01]{Dav01}
Kenneth~R. Davidson.
\newblock Free semigroup algebras. {A} survey.
\newblock In {\em Systems, approximation, singular integral operators, and
  related topics ({B}ordeaux, 2000)}, volume 129 of {\em Oper. Theory Adv.
  Appl.}, pages 209--240. Birkh\"{a}user, Basel, 2001.

\bibitem[DP98]{davidson-pitts-NP-1998}
K.R. Davidson and D.R. Pitts.
\newblock Nevanlinna-{P}ick interpolation for non-commutative analytic
  {T}oeplitz algebras.
\newblock {\em Integral Equations Operator Theory}, 31:321--337, 1998.

\bibitem[EHK78]{evans-hoegh-krohn-1978}
David~E. Evans and Raphael H{\o}egh-Krohn.
\newblock Spectral properties of positive maps on {$C\sp*$}-algebras.
\newblock {\em J. London Math. Soc. (2)}, 17(2):345--355, 1978.

\bibitem[JM19]{jury-martin-2019}
Michael~T. Jury and Robert T.~W. Martin.
\newblock Factorization in weak products of complete pick spaces.
\newblock {\em Bull. Lond. Math. Soc.}, 51(2):223--229, 2019.

\bibitem[KVV14]{KVV-tome}
Dmitry~S. Kaliuzhnyi-Verbovetskyi and Victor Vinnikov.
\newblock {\em Foundations of free noncommutative function theory}, volume 199
  of {\em Mathematical Surveys and Monographs}.
\newblock American Mathematical Society, Providence, RI, 2014.

\bibitem[Pasa]{pascoe-2019-committee}
James~Eldred Pascoe.
\newblock Committee spaces and the random column-row property.
\newblock {\em arXiv:1904.11129}.

\bibitem[Pasb]{pascoe-2019-QPF}
James~Eldred Pascoe.
\newblock The outer spectral radius and dynamics of completely positive maps.
\newblock {\em arXiv:1905.09895}.

\bibitem[Pas19]{pascoe-2019-alg}
J.~E. Pascoe.
\newblock An elementary method to compute the algebra generated by some given
  matrices and its dimension.
\newblock {\em Linear Algebra Appl.}, 571:132--142, 2019.

\bibitem[Pau02]{paulsen-book}
Vern Paulsen.
\newblock {\em Completely bounded maps and operator algebras}, volume~78 of
  {\em Cambridge Studies in Advanced Mathematics}.
\newblock Cambridge University Press, Cambridge, 2002.

\bibitem[Pop06]{popescu-2006}
Gelu Popescu.
\newblock Free holomorphic functions on the unit ball of {$B(H)^n$}.
\newblock {\em J. Funct. Anal.}, 241(1):268--333, 2006.

\bibitem[Sha15]{shalit-2013}
O.~Shalit.
\newblock Operator theory and function theory in {D}rury--{A}rveson space and
  its quotients.
\newblock In {\em Handbook of Operator Theory}, pages 1125--1180. 2015.

\bibitem[Tre04]{trent-2004}
Tavan~T. Trent.
\newblock A corona theorem for multipliers on {D}irichlet space.
\newblock {\em Integral Equations Operator Theory}, 49(1):123--139, 2004.

\end{thebibliography}

\end{document}